\documentclass{amsart}
\usepackage{graphicx}
\usepackage{amssymb}
\usepackage{enumitem}
\usepackage{comment}
\usepackage{xcolor}

\newtheorem{thm}{Theorem}[section]

\newtheorem{lem}[thm]{Lemma}
\newtheorem{prop}[thm]{Proposition}
\theoremstyle{definition}

\theoremstyle{remark}
\newtheorem{rem}[thm]{Remark}

\newcommand{\re}{\mathbb R}
\newcommand{\pd}{\partial}
\newcommand{\ep}{\varepsilon}

\newcommand{\eqal}[1]{\begin{equation}\begin{aligned}#1\end{aligned}\end{equation}}
\newcommand{\es}[1]{\begin{align*}#1\end{align*}}

\usepackage{etoolbox}

\appto\appendix{\counterwithin{equation}{section}}

\title[Semi-convex viscosity solutions]{Semi-convex viscosity solutions of the special Lagrangian equation}
\author{Connor Mooney and Ravi Shankar}
\date{\today}

\begin{document}

\begin{abstract}
We prove smoothness and interior derivative estimates for viscosity solutions to the special Lagrangian equation with almost negative phases and small enough semi-convexity. We show by example that the range of phases we consider and the semi-convexity condition are sharp. As an application, we find a new Liouville theorem for entire such solutions of the special Lagrangian equation with subcritical phase. We also find effective Hessian estimates with exponential dependence, which we show to be optimal.
\end{abstract}

\maketitle

\section{Introduction}

A function $u$ on a domain in $\re^n$ solves the special Lagrangian equation if
\eqal{
\label{eq}
F(D^2u) := \sum_{i=1}^n\tan^{-1}\lambda_i = \Theta,
}
where $\lambda_i$ are the eigenvalues of $D^2u$, and the phase $\Theta\in(-n\pi/2,n\pi/2)$ is a constant.  In the case that $u\in C^{1,1}$, the gradient graph $\{(x,Du(x))\} \subset \re^n\times\re^n$ is minimal, by Harvey and Lawson's calibration argument \cite{HL}.

\smallskip
The regularity question for viscosity solutions of the fully nonlinear, degenerate elliptic PDE \eqref{eq} is delicate. In two dimensions, regularity was demonstrated by Heinz in the 1950's \cite{H}.  The Hessian estimate also follows from Gregori in 1994 \cite{G}.  For critical and supercritical $|\Theta|\ge (n-2)\pi/2$ phases or convex solutions, the level set of $F$ is convex \cite{Y3}.  For such phases in general dimension \cite{WdY2}, or for convex viscosity solutions \cite{CSY}, it is known that $u$ is analytic.  However, the first-named author and Savin \cite{MS} constructed Lipschitz but not $C^1$, semi-convex viscosity solutions in dimension $n \geq 3$ whose gradient graphs are non-minimal.  Earlier, $C^{1,\alpha}$ singular semi-convex viscosity solutions were constructed in dimension $n \geq 3$ by Nadirashvili-Vladut and Wang-Yuan \cite{NV1,WdY1}. The gradient graphs of the latter examples are analytic and minimal.

\smallskip
In this paper, we ask when the gradient graph of a semi-convex viscosity solution is minimal.  It turns out that for almost negative, subcritical phases and small enough semi-convexity, we can show that viscosity solutions are analytic, with a linear exponential bound for the Hessian. We show that the almost-negativity of the phase and the semi-convexity are sharp, by generalizing the examples from \cite{MS}. We also show that the linear exponential dependence in the Hessian bound is optimal.

\begin{thm}[Regularity]\label{thm:reg}
Let $\Theta \in (-(n-2)\pi/2,\,\pi/2)$
and define
\eqal{
\label{theta}
\theta := \frac{\pi/2 - \Theta}{n-1} \in (0,\,\pi/2).
}
Assume that $u$ is a viscosity solution in $B_1 \subset \mathbb{R}^n$ to
$$F(D^2u) = \Theta,$$
and satisfies in addition that
\begin{equation}\label{SemiconvexBound}
u + \frac{1}{2}\tan(\theta)|x|^2 \text{ is convex}.
\end{equation}
Then $u$ is analytic, and moreover we have for $k \geq 2$ that
\begin{equation}\label{HessianBound}
|D^ku(0)| \leq e^{C(n,\,k,\,\Theta)\left(1 + \|Du\|_{L^{\infty}(B_1)}\right)}.
\end{equation}
\end{thm}

We recall that regularity is true in the critical and supercritical cases $|\Theta| \geq (n-2)\pi/2$, and for convex solutions. It is natural to ask whether any negative lower bound on the Hessian will imply regularity for the subcritical phases $\Theta \in [\pi/2,\, (n-2)\pi/2)$, and whether the lower bound we assumed in Theorem \ref{thm:reg} can be lowered. We show that the answer to both questions is ``no":

\begin{thm}[Sharpness of assumptions]\label{thm:Ex}
For any $\Theta \in [\pi/2,\, (n-2)\pi/2)$ and $\epsilon > 0$, there exist singular viscosity solutions to
$$F(D^2u) = \Theta$$
such that $u + \epsilon |x|^2$ is convex. 

For any $\Theta \in (-(n-2)\pi/2,\,\pi/2)$ and $\epsilon > 0$, there exist singular viscosity solutions to
$$F(D^2u) = \Theta$$
such that $u + \frac{1}{2}(\tan\theta + \epsilon)|x|^2$ is convex.

In both cases, the examples may be taken to be Lipschitz but not $C^1$, and to have non-minimal gradient graph.
\end{thm}

It is also natural to ask whether the exponential dependence on $\|Du\|_{L^{\infty}}$ in the effective bound (\ref{HessianBound}) can be improved. We supply a negative answer:

\begin{thm}[Sharpness of effective bound]\label{thm:Ex2}
There exist smooth solutions to (\ref{eq}) in $B_1 \subset \mathbb{R}^n$ showing that an effective bound for $|D^2u(0)|$ depending exponentially on $\|Du\|_{L^{\infty}(B_1)}$ is the best one can expect in either of the cases:
\begin{enumerate}
\item $u$ satisfies the conditions of Theorem \ref{thm:reg},
\item $u$ is convex and $\Theta \geq \frac{\pi}{2}$.
\end{enumerate}
\end{thm}

Finally, as an application of Theorem \ref{thm:reg}, we find a new Liouville theorem for entire viscosity solutions.  
\begin{thm}[Liouville]\label{thm:liou}
    Let $\Theta\in(-(n-2)\pi/2,\pi/2)$, and $\theta(n,\Theta)$ be as in \eqref{theta}.  Assume that $u$ is a viscosity solution on $\re^n$ to 
    $$
    F(D^2u)=\Theta,
    $$
    and satisfies $\tan\theta$-convexity \eqref{SemiconvexBound}.  Then $u$ is a quadratic polynomial.
\end{thm}

Ordinarily, one would establish Theorem \ref{thm:reg} by combining the classical solvablity of the Dirichlet problem with interior a priori estimates.  Classical $C^\infty$ solvability of the Dirichlet problem is only established for critical or supercritical phases \cite{CNS}.  Not only is Theorem \ref{thm:reg} about semi-convex solutions, which would not be preserved under smooth Dirichlet approximations, but the phases are subcritical.  An additional issue is that of finding a Jacobi inequality for the conditions of Theorem \ref{thm:reg}, which is needed for most of the current methods for finding interior a priori estimates, which use Trudinger's approach for the minimal surface gradient estimate.

\smallskip
To prove Theorem \ref{thm:reg}, one might hope to adapt the argument in \cite{CSY}, which works for convex viscosity solutions, to semi-convex viscosity solutions.  In that work, the Lewy-Yuan rotation $(x,y)\mapsto (x\cos\phi +y\sin\phi,-x\sin\phi+y\cos\phi)$ of the gradient graph is defined for general semi-convex functions.  The gradient graph achieves bounded slopes, allowing for the regularity of the rotated solution, which also solves \eqref{eq}.  A constant rank theorem for the Hessian then passes the regularity back to the original solution.  However, the preservation of subsolutions was only achieved for convex viscosity solutions.  In general, for equations without convex level set, the preservation of subsolutions is false.  The semi-convex viscosity solutions in \cite{MS} have non-minimal gradient graph.  Their rotations do not solve \eqref{eq} on an open set: the image of the solution's singularity.

\smallskip
The first observation behind Theorem \ref{thm:reg} is that under the phase and semi-convexity conditions assumed, we can establish the preservation of subsolutions under rotation, Proposition \ref{prop:rot}.  The idea behind these conditions is that the inequality $F(D^2\bar u)\ge\Theta$ can be satisfied inside the image of the singularity if $\Theta$ is not too positive, and the angles $\tan^{-1}\bar\lambda_i$ comprising $F$ are not too negative. Remarkably, by the examples in Theorem \ref{thm:Ex}, this argument is sharp.  An extension to solutions with only codimension two and higher singularities is outlined in Remark \ref{rem:codim}, although the subsequent regularity question appears more subtle.

\medskip
It appears difficult to establish subsolution preservation in other contexts.  Indeed, the equation for the rotation here has convex sublevel set, so the original equation \eqref{eq} is inverse convex, while Alvarez-Lions-Lasry \cite{ALL} assumed $C^2$ to show subsolution preservation for their inverse-convex equations.  It is an interesting question whether assuming $C^1$ for semi-convex viscosity solutions is enough for subsolution preservation.  This might allow one to distinguish between the examples in \cite{MS} and the examples in \cite{NV1,WdY1}.

\smallskip
The second technical aspect of Theorem \ref{thm:reg} is the constant rank part of the argument.  If the equation \eqref{eq} has convex sublevel set, then the rotated equation is inverse convex, the condition by \cite{ALL} which allows for the constant rank theorem of Caffarelli-Guan-Ma \cite{CGM}.  In \cite{CSY}, convex viscosity solutions are assumed, but level set convexity fails in the setting of Theorem \ref{thm:reg}.  The rotated equation has negative supercritical phase with convex sublevel set, so there is a constant rank theorem for the minimum eigenvalues.  Although this is insufficient, it turns out that there is eigenvalue rigidity.  Indeed, by \eqref{theta}, it is possible to rewrite \eqref{eq} as 
\eqal{
\label{altEq}
\pi/2-\tan^{-1}\lambda_1=\sum_{i=2}^n(\tan^{-1}\lambda_i+\theta).
}
At blowup points of the maximum eigenvalue $\lambda_1$, the $\tan(\theta)-$convexity forces the minimum eigenvalues to saturate their lower bound.  The constancy of the maximum eigenvalue follows from constant rank for the minimum eigenvalues and the constant coefficient equation \eqref{eq}.  The singular solutions by \cite{NV1,WdY1} fail the $\tan(\theta)$-convexity condition, except at the singularity, which shows the sharpness of this argument.  It is interesting to note that by \eqref{altEq}, if our viscosity solution is slightly more convex, $D^2u>-\tan(\theta)I$, then one can show that $u$ is $C^{1,1}$, without need for the regularity of its gradient graph.

\smallskip
This eigenvalue rigidity can be understood from the area decreasing condition \cite{WMT}, when the area integrand is convex.  Applying a formal partial rotation of \cite{WdY1}, it is possible to locally map smooth such semi-convex solutions with almost negative phase to convex solutions with phase $\Theta=\pi/2$, which are area-decreasing solutions that are convex.  The constant rank idea would also be valid for such solutions.  A different partial rotation yields the two-convexity condition in \cite{TTW1}.  It is unclear how to formulate viscosity solutions satisfying the area decreasing condition or two-convexity condition, and the partial rotation is not clear for viscosity solutions.  One can view Theorem \ref{thm:reg} as a formal, viscosity realization of regularity for a partial rotation of the area decreasing condition and the two-convexity condition.

\smallskip
Another aspect of Theorem \ref{thm:reg} is the method for deriving the (sharp) effective estimate.  Earlier works \cite{WdY2,WY2,WMT,CWY,WY3,WY4,Z1,Z2} 
derived exponential type estimates using the Trudinger \cite{T} type argument for the gradient estimate for minimal surfaces, with a pointwise argument in \cite{WY1} and compactness arguments in \cite{L3,Sh}.  That is, a Jacobi inequality is established, and the mean value and Sobolev inequalities lead to an integration by parts argument.  However, the Jacobi inequality is less clear under the conditions of Theorem \ref{thm:reg}.  Jacobi inequalities can be valid for certain semi-convex solutions \cite{WY1,L2}, but the Hessian lower bound $-\tan(\theta)$ here can be arbitrarily large.  We circumvent the Jacobi inequality by quantifying the constant rank theorem argument.  Using the weak Harnack inequality and a quantified chain of ball argument, we can get an explicit estimate in rotated space.  This idea is similar to the one employed in \cite{D}, where a lattice is used on the rotated graph.  A quantitative version of the standard constant rank theorem is given earlier in \cite{SW}.  Legendre transform connections of constant rank to regularity are established in \cite{BS}.

\smallskip
Theorem \ref{thm:Ex2} shows that our linear exponential estimate is sharp.  Our example in Section \ref{Ex2Proof} is explicit and arises from a small harmonic function with separated variables.  Applying the partial Legendre transform in the direction with convexity yields a solution of the Monge-Amp\`ere equation, or $\Theta=\pi/2$ in dimension two.  Further rotations and additions by quadratics embed this example in higher dimensions and other phases.  Therefore, linear exponential dependence is the best estimate that can be expected in these situations; see Remark \ref{rem:sharp}.  It was earlier observed in \cite{WdY2} that applying the Heinz transformation to Finn's minimal surface example with optimal exponential dependence would yield a similar solution to the Monge-Amp\`ere equation.  This solution would also have optimal dependence.  It is interesting whether the critical phase special Lagrangian equation has exponential dependence, since this rotation argument is not possible.  There is an $\exp|Du|^3$ dependence in dimension three and critical phase $\Theta=\pi/2$ \cite{WY3,WY4}.

\smallskip
Caffarelli showed a Liouville theorem for convex viscosity solutions of the Monge-Amp\`ere equation, with the smooth case by Jorgens-Calabi-Pogorelov and Cheng-Yau.  Theorem \ref{thm:liou} concerns the semi-convex viscosity solutions of a special Lagrangian equation with subcritical phase.  For convex viscosity solutions, the Liouville theorem is known by \cite{CSY,Y2}.  For supercritical phase, such a Liouville theorem follows from the regularity of viscosity solutions and the Lewy-Yuan rotation to an equation with convex level set \cite{Y3}.  Theorem \ref{thm:liou} proceeds in a similar way, once regularity is established.  Regarding the connection with the area decreasing condition, there is also a Liouville theorem for entire smooth solutions by Tsui and Wang \cite{TW}, but it is not clear whether our theorem follows from theirs, since the equivalence by partial rotation is only local.

\smallskip
In general, the Liouville theorem is false for subcritical phases and non-convex solutions.  For critical phase $\Theta=\pi/2$ in dimension three, there are exponential-type solutions by Warren \cite{W1}.  For subcritical phase $\Theta=0$, there are cubic-type solutions by Li \cite{L1}.  For general subcritical phases in dimension three, there are non-splitting entire solutions by Li \cite{L4}.  For dimension  $n\ge 4$ and subcritical phases $\Theta\in [-\pi/2,\pi/2]$, non-polynomial solutions are constructed by Tsai, Tsui, and Wang \cite{TTW2}.  

\smallskip
For smooth solutions with subcritical phases, Ogden and Yuan \cite{OY} showed a Liouville theorem for semi-convex solutions satisfying $\tan^{-1}\lambda_{min}\ge (\Theta-\pi)/n$, with no restriction on the phase.  In Theorem \ref{thm:liou}, the viscosity solutions must satisfy the stronger conditions that $\tan^{-1}\lambda_{min}\ge (\Theta-\pi/2)/(n-1)$ and $-(n-2)\pi/2<\Theta<\pi/2$.  However, by Theorem \ref{thm:Ex}, such stronger conditions are necessary for interior regularity, so it is unclear how to generalize Ogden and Yuan's theorem to entire viscosity solutions.

\smallskip
The paper is organized as follows. In Section \ref{Preliminaries}, we discuss some preliminary results concerning semi-convex functions, Legendre transforms, and rotations. In Section \ref{RegProof} we prove Theorem \ref{thm:reg}, and we derive the Liouville-type result Theorem \ref{thm:liou} as a result. In Section \ref{Ex1Proof} we construct the examples from Theorem \ref{thm:Ex}. Finally, in Section \ref{Ex2Proof} we construct the examples from Theorem \ref{thm:Ex2}.

\section*{Acknowledgements}
C. Mooney was supported by a Simons Fellowship and NSF grant DMS-2143668.  R. Shankar thanks Yu Yuan for comments and Guido De Philippis for a stimulating discussion.


\section{Preliminaries}\label{Preliminaries}

\subsection{Viscosity solutions}

Let $F(M)$ be a continuous function on the symmetric matrices that is elliptic, i.e. $F(M + N) \geq F(M)$ for all $M$ and any $N \geq 0$. We say that $u$ is a viscosity subsolution of the fully nonlinear elliptic PDE $F(D^2u)=0$, if $F(D^2P)\ge 0$ for each quadratic $P$ touching $u$ from above near a point, or $P(x_0)=u(x_0)$ with $P\ge u$ near $x_0\in\Omega$; see \cite[Proposition 2.4]{CC}.  A smooth viscosity subsolution satisfies $F(D^2u)\ge 0$ pointwise.  A supersolution satisfies the reverse inequality, and a solution is both a subsolution and a supersolution.

\smallskip
Now suppose that $F$ is locally Lipschitz and locally uniformly elliptic, i.e. for all $R > 0$, there is some $\lambda(R) > 0$ such that the eigenvalues of $(\partial F/\partial M_{ij})(M)$ are in $[\lambda,\,\lambda^{-1}]$ whenever $|M| \leq R$. Assume that $u$ is a viscosity solution to $F(D^2u) = 0$. If $F$ is also a concave function on the symmetric matrices, then the Evans-Krylov estimate, see also \cite{CC} for $C^{1,1}$ solutions, gives a $C^{2,\alpha}$ estimate for $u$ on a domain in terms of the $L^\infty$ norm of the Hessian on a larger domain. If instead $|D^2u| \le C$ and $F$ is smooth with $\{F > 0\}$ convex, then $u$ also viscosity solves a concave uniformly elliptic equation $G(D^2u)=0$. The function $G$ can be obtained by minimizing over those linear functions tangent to $F$ on $\{F = 0\} \cap \{|M| \leq C\}$.  In the supercritical phase case for \eqref{eq}, $G$ is given in \cite{CPW} as $G(D^2u)=-\exp(-AF(D^2u))$ for large $A$.

\medskip
If $u$ is a smooth solution of $F(D^2u)=0$, a smooth elliptic PDE with $\{F > 0\}$ convex, then the double derivatives $u_{ee}$ where $e\in S^{n-1}$, are subsolutions of the linearized operator,
$$
\frac{\pd F}{\pd u_{ij}}(D^2u)\frac{\pd^2}{\pd x^i\pd x^j}(u_{ee})=-\frac{\pd^2 F}{\pd u_{ij}\pd u_{k\ell}}(D^2u)u_{eij}u_{ek\ell}\ge 0.
$$
Here, implied summation is used.  The maximum of subsolutions of  is a viscosity subsolution \cite{CC}, so by the variational principle, the maximum eigenvalue $\lambda_{1}$ of the Hessian $D^2u$ is a viscosity subsolution.  More generally, the convex combinations $\lambda_1+\lambda_2+\cdots+\lambda_k$ of the largest eigenvalues of the Hessian are also viscosity subsolutions.

\medskip
In the case of the special Lagrangian equation
$$
F(D^2u)=\sum_{i=1}^n\tan^{-1}\lambda_i=\Theta=const.,
$$
the linearized operator is equivalent to the Laplace-Beltrami operator $\Delta_g$ with induced metric $g=dx^2+dy^2$ on gradient graph $(x,y)=(x,Du(x))$:
$$
\frac{\pd F}{\pd u_{ij}}(D^2u)\frac{\pd^2}{\pd x^i\pd x^j}=\Delta_g=\frac{1}{\sqrt{\det g}}\partial_i(\sqrt{\det g}g^{ij}\partial_j).
$$

\subsection{Lipschitz semi-convex functions}
We recall the notion of a gradient graph for Lipschitz functions.  Let $u$ be a $C^{0,1}$ convex function on open set $\Omega\subset\re^n$.  If a tangent plane with slope $p\in\re^n$ touches $u$ at $x_0\in\Omega$ from below, then $p$ is called a subgradient of $u$, denoted $p\in\pd u(x_0)$, with $\pd u(x_0)$ the subdifferential, or collection of subgradients.  A subgradient $p$ of a \textbf{strictly convex function} touches it at a unique point $x_0$.  Each subdifferential is a closed, convex set.  Moreover, the \textbf{gradient image} $\pd u(U)$ of an open set $U$ is open, if $u$ is strictly convex.  The gradient images of disjoint compact sets are separated, if $u$ is strictly convex.  If it is $-L$-convex for some $L>0$, or $D^2u\ge LI$ for $u\in C^{2}$, then the gradient map $x\mapsto \pd u(x)$ is \textbf{distance-increasing}, with
\eqal{
|\bar x_1-\bar x_2|^2\ge L|x_1-x_2|^2,\qquad \bar x_i\in \pd u(x_i).
}
For $K\in\re$, we write $D^2u\ge KI$ if $u-K|x|^2/2$ is convex, and similarly for $D^2u\le KI$.  We say that $D^2u>KI$ if there exists $\ep>0$ such that $D^2u\ge (K+\ep)I$, and similarly with $D^2u<KI$, and $D^2u<\infty$ if $D^2u\le KI$ for some $K$.  We recall that $u$ is $C^{1,1}$ if $-KI\le D^2u\le CI$ \cite[Proposition 1.2]{CC}.  

\smallskip
The gradient image is the vertical projection of the \textbf{``gradient graph"}, the subset $\{(x,\pd u(x)),x\in\Omega\}$ of $\Omega\times\re^n$ that functions like a multi-valued graph of the gradient of $u$ over $\Omega$.  For example, if $u=|x_1|$ on $\re$, then $\pd u(x)=\{\pm 1\}$ for $\pm x>0$, while $\pd u(0)=[-1,1]$, so the gradient graph has a vertical step.

\medskip
The gradient graph is also defined for semi-convex functions.  
If $D^2u\ge -KI$, then we say that $p\in\re^n$ is a subgradient of $u$ at $x_0$ if $p+Kx_0$ is a subgradient of $u+K|x|^2/2$ at $x_0$.  The collection of such subgradients is the subdifferential of $u$, which is just a shift: $\pd u(x_0)=\pd (u+K|x|^2/2)(x_0)-Kx_0$.  In a similar way, we define the gradient image $\pd u(\Omega)$ and gradient graph of a semi-convex function $u$ on $\Omega$.

\subsection{Touching reversal under Legendre transform}

This section fleshes out some details indicated in \cite[Proposition 3.2]{CSY}.  The main technical results from this section verify that the touching of two functions is reversed under the Legendre transform, or the inverse Legendre transform.  

\medskip
For convex $f(x)$ on an open set $\Omega$, its Legendre transform over $\Omega$ is a convex function defined by
\eqal{
\label{Leg}
f^*(\bar x)=\sup_{x\in\Omega}(\bar x\cdot x-f(x)),\qquad \bar x\in \pd f(\Omega).
}
The Legendre transform of a strongly convex quadratic $Q(x)=\frac{1}{2}\langle x,Mx\rangle$ on $\re^n$ is $\frac{1}{2}\langle \bar x,M^{-1}\bar x\rangle$.  The transform of $|x|$ on $\re^n$ is $0$ on $\overline{B_1(0)}$ and $\infty$ outside, and the transform of
the latter function is $|x|$.

\smallskip
The Legendre transform restricts well: the transform over a smaller domain $\Omega'\subset\Omega$ is equal to $f^*$ on $\pd f(\Omega')$.  On the other hand, if $f$ is continuous on $\overline\Omega$ and $\Omega$ is convex, one can extend $f$ to $+\infty$ outside $\overline\Omega$ to make it convex, and its Legendre transform extends from $\pd f(\Omega)$ to $\re^n$, with $\pd f(\re^n)=\pd f(\overline{\Omega})=\re^n$.  What happens is that the remaining tangent planes touch the graph over the boundary $\pd\Omega$.  

\smallskip
This extension to $\re^n$ is closed (equal to the greatest lower semi-continuous function $h\le f$), in the terminology of \cite[pg 52]{R}. A closed convex function $f$ is also equal to the supremum of the planes $p\le f$ \cite[Theorem 12.1, pg 102]{R}.

\smallskip
If $f$ is a closed convex function, then the Legendre transform is involutive, $f^{**}=f$ \cite[Theorem 12.2, pg 104]{R}.  To satisfy the closed condition, we will use the extension outlined above.  It is also useful to assume that $f$ is strictly convex; this is valid in the applications to rotation.

\smallskip
The transform is also \textbf{order reversing}: if $g\le f$ on $U$, then $g^*\ge f^*$ on $\pd f(U)\cap \pd g(U)$.  If $g$ touches $f$ from below at $x_0 \in U$, and $\pd f(U) \cap \pd g(U)$ contains an open neighborhood of a point $p \in \pd g(x_0)$, then $g^*$ touches $f^*$ from above at $p$. 
We use this observation in the following lemma.  

\begin{lem}[Touching reversal of transforms]
\label{lem:Leg}
Let $Q$ be a uniformly convex quadratic polynomial, and $g$ be a strictly convex function on $B_1$.  

\smallskip
\textbf{Part 1:} Assume that $Q$ touches $g$ from above (below) near $x_0 \in B_1$. Then $Q^*$ touches $g^*$ from below (above) near $DQ(x_0)$.

\smallskip
\textbf{Part 2:} Assume that $Q^*$ touches $g^*$ from below near $\bar x_0 \in \partial g(x_0),\, x_0 \in B_1$. Then $Q$ touches $g$ from above near $x_0$.

\smallskip
\textbf{Part 3:} Assume that $Q^*$ touches $g^*$ from above near $\bar x_0 \in \partial g(x_0),\, x_0 \in B_1$. Then $Q$ touches $g$ from below near $x_0$.

\smallskip
\textbf{Part 4:} If $D^2g\ge KI>0$, then $D^2g^*\le K^{-1}I$.  If $D^2g^*\ge KI$, then $D^2g\le K^{-1}I$.

\smallskip
\textbf{Part 5:} If $D^2g>0$, and $g$ is twice differentiable at $x_0$, then $g^*$ is twice differentiable at $\bar x_0=Dg(x_0)$, with 
$$
D^2g^*(\bar x_0)=D^2g(x_0)^{-1}.
$$
If $D^2g^*>0$, and $g^*$ is twice differentiable at $\bar x_0$, then $g$ is twice differentiable at $x_0$, with the above formula.
\end{lem}

\begin{proof}
\textbf{Part 1:} First, $D Q(x_0) \in \partial g(x_0)$, with equality when $Q$ touches from above. We have that $Q$ touches $g$ above or below in $B_r(x_0)$ for some $r > 0$. By strict convexity, $\partial g(B_r(x_0))$ is open. Order reversing completes the proof.

\smallskip
\textbf{Part 2:} The gradient of $Q^*$ acts like an open map.  Without loss of generality we may extend $g$ to $+\infty$ outside $\overline{B_R}$, with $|x_0| < R < 1$, so that $g^{**} = g$. By strict convexity, touching still happens near $\bar x_0$, say in $B_r(\bar x_0)$ for some small $r > 0$, where $g^*$ is $C^1$ and $D g^*(\bar x_0) = D Q^*(\bar x_0) = x_0$ \cite[Theorem 26.3, pg 253]{R}. By order reversing, we just need to show that $D g^*(B_r(\bar x_0))$ contains a neighborhood of $x_0$. This follows from touching below and the uniform convexity of $Q^*$: for $h > 0$ small, 
$$\{g^*(\bar x) < g^*(\bar x_0) + x_0 \cdot (\bar x - \bar x_0) + h\} \subset \{Q^*(\bar x) < g^*(\bar x_0) + x_0 \cdot (\bar x - \bar x_0) + h\} \subset B_r(\bar x_0).$$
Since $D g^*(\{g^*(\bar x) < g^*(\bar x_0) + x_0 \cdot (\bar x - \bar x_0) + h\})$ contains a neighborhood of $x_0$, we are done.

\smallskip
\textbf{Part 3:} Now the gradient may not be an open map.  Make the same extension as above. Subtracting a linear function and translating, we may assume that $x_0 = \bar x_0 = 0$ and that $g(0) = 0$. For some $r > 0$ we have that the uniformly convex $Q^*$ touches $g^*$ from above at $0$ in $B_r$. Moreover, $\partial g(0)= \{g^* = 0\} $ is compact, hence $\{g^* < h\}$ is bounded for $h > 0$. We claim that for $h > 0$ small, $Q^* \geq g^*$ in $\{g^* < h\}$. Indeed, this is automatic in $B_r$, while outside $B_r$, we have in this set that $g^* < h < cr^2 < Q^*$ for $h$ sufficiently small. Since $\partial g^*(\{g^* < h\})$ contains a neighborhood of zero, order reversing completes the proof.

\smallskip
\textbf{Part 4:} We use Part 1.  If $g-K|x|^2/2$ is convex, then for any subgradient $p\in \pd g(x_0)$, there is a quadratic $Q$ with $D^2Q=KI$ and $DQ(x_0)=p$ which touches $g$ from below.  This means such a $Q^*$ touches $g^*$ from above at all such subgradients, hence $D^2g^*\le K^{-1}I$ on all of $\pd g(B_1)$ \cite[Proposition 1.6]{CC}.  The second case follows from Part 2.

\smallskip
\textbf{Part 5:} If $g$ is twice differentiable at $x_0$ and $D^2g\ge KI>0$, then uniformly convex quadratics $Q$ with $D^2Q=D^2g(x_0)\pm \ep I$ touch $g$ from above and below, so Part 1 of Lemma \ref{lem:Leg} shows that quadratics $Q^*$ with $D^2Q^*=(D^2g(x_0)\pm \ep I)^{-1}$ touch $g^*$ from below and above, so $g^*$ is twice differentiable at $Dg(x_0)$, with $D^2g^*(Dg(x_0))=D^2g(x_0)^{-1}$.  By Parts 2 and 3, the same argument works if $g$ is strictly convex, and $g^*$ is twice differentiable and strongly convex.
\end{proof}

\subsection{CSY rotation and touching preservation}
The gradient graph $(x,\pd u(x))$ is continuous in the standard metric on $\re^n\times\re^n$ by \cite[Corollary 24.5.1]{R}.  In the case of a convex potential, one sees it is a Lipschitz submanifold of $\re^n\times\re^n$ by the downward $\pi/4$ rotation $(x,y)\mapsto (x+y,-x+y)/\sqrt 2$ (rotation upwards of the axes), after which it is a Lipschitz graph $(\bar x,\bar F(\bar x))$; \cite[Proposition 1.1]{AA}.  

\smallskip
Since this rotation preserves symplectic form $dx\wedge dy$, the graph should still be Lagrangian, and expressible as a gradient graph $(\bar x,D\bar u(\bar x))$.  In the smooth case, the potential $\bar u$ is obtained as the Legendre-Lewy-Yuan transform \cite{Y2}.  In the $C^1$ potential case, an explicit expression for $\bar u(\bar x)$ is given by Warren \cite{W2} and Chen-Warren \cite{CW}.  For Lipschitz potentials $u(x)$, Chen-Shankar-Yuan \cite{CSY} expressed $\bar u$ in terms of a conjugation of the Legendre transform by scalings and adding quadratics:
$$
\bar u(\bar x)=\frac{1}{2}|\bar x|^2-\sqrt 2\left[\frac{1}{\sqrt 2}\left(u(x)+\frac{1}{2}|x|^2\right)\right]^*(\bar x),\qquad \bar x\in \pd \left(\frac{u+|x|^2/2}{\sqrt 2}\right)(\Omega),
$$
where $f^*(\bar x)$ is the Legendre transform.  

\smallskip
We now recall the \cite{CSY} rotation for general $C^{0,1}$ semi-convex functions.  Let us assume $D^2u\ge-\tan(\theta)I$ for some $\theta\in (0,\pi/2)$, or $\pi/2\ge \tan^{-1}\lambda_i \ge -\theta$.  We can rotate down by an angle $0<\phi<\pi/2-\theta$.  Let $s=\sin\phi$ and $c=\cos\phi$.  Then the rotation $\bar x+i\bar y=e^{-i\phi}(x+i\pd u(x))$ of the gradient graph, or in terms of subgradients $y\in \pd u(x)$,
\es{
\bar x&=cx+sy,\\
\bar y&=-sx+cy,
}
can be realized as a Lipschitz, gradient graph $(\bar x,D\bar u(\bar x))$, with potential involving the Legendre transform:
\eqal{
\label{rot}
\bar u(\bar x)=\frac{c}{2s}|\bar x|^2-\frac{1}{s}\left(su+c\frac{|x|^2}{2}\right)^*(\bar x),\qquad \bar x\in \pd \left(su+c|x|^2/2\right)(x).
}
Since $\cot\phi>\tan\theta$, it follows that $su+c|x|^2/2$ is uniformly convex, so touching reversal Lemma \ref{lem:Leg} is valid for the Legendre transform part.  In terms of the rotation $\bar u$, there is \textbf{touching preservation}.  Namely, we apply the Legendre transform touching reversal Lemma \ref{lem:Leg} to
$$
U(x):=su(x)+c|x|^2/2,\qquad U^*(\bar x)=-s\bar u(\bar x)+c|\bar x|^2/2,
$$
and the associated quadratics derived from $Q$ and $\bar Q$.  We record the corresponding results below.

\begin{prop}[Touching preservation of rotation]
\label{prop:rot-touch}
    On a domain $B_1\subset\re^n$, let $D^2u,D^2Q>-\tan(\pi/2-\phi)I$ for some $\phi\in(0,\pi/2)$, with $Q$ a quadratic polynomial.  Let $\bar u(\bar x),\bar Q(\bar x)$ be the $\phi$-rotations \eqref{rot}.  Let $x_0\in B_1$ and $\bar x_0\in \pd (su+c|x|^2/2)(x_0)$.
    
\smallskip
    \textbf{Part 1:} Assume that $Q$ touches $u$ from below (above) near $x_0$.  Then $\bar Q$ touches $\bar u$ from below (above) near $D(sQ+c|x|^2/2)(x_0)$.

\smallskip
    \textbf{Part 2:} Assume that $\bar Q$ touches $\bar u$ from above near $\bar x_0$.  Then $Q$ touches $u$ from above near $x_0$.

\smallskip
    \textbf{Part 3:} Assume that $\bar Q$ touches $\bar u$ from below near $\bar x_0$.  Then $Q$ touches $u$ from below near $x_0$.
    
\smallskip
    \textbf{Part 4:} The rotation $\bar u$ is $C^{1,1}$: if $D^2u\ge\tan(-\theta)I$ for some $\theta\in(0,\pi/2-\phi)$, then
    $
    \tan(-\theta-\phi)I\le D^2\bar u\le \tan(\pi/2-\phi)I.
    $
    If $D^2\bar u<\tan(\pi/2-\phi)$, then $D^2u<\infty$.

\smallskip
    \textbf{Part 5:} If $u$ is twice differentiable at $x_0$, then $\bar u$ is twice differentiable at $\bar x_0$.  If $\bar u$ is twice differentiable at $\bar x_0$ and $D^2\bar u<\tan(\pi/2-\phi)$, then $u$ is twice differentiable at $x_0$.  In each case, there is the formula
$$
D^2\bar u(\bar x_0)=(-sI+cD^2u)(cI+sD^2u)^{-1}(x_0).
$$
\end{prop}

If $\lambda_i$ and $\bar\lambda_i$ are eigenvalues of $D^2u(x)$ and $D^2\bar u(\bar x)$ respectively, then 
\eqal{
\label{angles}
\tan^{-1}\bar\lambda_i=\tan^{-1}\lambda_i-\phi.
}
In the smooth case, this is the well known Legendre-Lewy-Yuan rotation by general angle \cite{WY2}.  For $C^1$ potentials $u(x)$, the potential is given by Warren \cite{W2} and Chen-Warren \cite{CW}.


\subsection{A bound on the non-maximal Hessian eigenvalues}
To conclude this section we record a simple bound on the Hessian eigenvalues of a function satisfying the conditions of Theorem \ref{thm:reg}:

\begin{lem}
If $u$ is a smooth function satisfying the hypotheses of Theorem \ref{thm:reg} with Hessian eigenvalues $\lambda_1 \geq \lambda_2 \geq ... \geq \lambda_n$, then
\begin{equation}\label{Lambda2}
-\tan(\theta) \leq \lambda_i < 1 \text{ for all } i \geq 2.
\end{equation}
\end{lem}
\begin{proof}
We have
\begin{align*}
2\tan^{-1}(\lambda_2) &\leq \tan^{-1}(\lambda_1) + \tan^{-1}(\lambda_2) \\
&= \Theta - \sum_{i >2} \tan^{-1}(\lambda_i) \\
&\leq \frac{1}{n-1}\Theta + \frac{n-2}{n-1}\frac{\pi}{2} \\
&< \frac{\pi}{2},
\end{align*}
where in the second inequality we used (\ref{theta}) and (\ref{SemiconvexBound}).
\end{proof}


\section{Proof of Theorem \ref{thm:reg} (Regularity) and Theorem \ref{thm:liou} (Liouville)}\label{RegProof}

\subsection{Preservation of semi-convex solutions under rotation}

If $u$ is such a semi-convex viscosity solution of $F(D^2u)=\Theta$ in $\Omega$, then in \cite[Proposition 3.2]{CSY}, it is shown that $\bar u$ in \eqref{rot} is a viscosity \textbf{supersolution} of 
\eqal{
F(D^2\bar u)=\Theta-n\phi,\qquad x\in \pd(su+c|x|^2/2)(\Omega).
}
For completeness, we recall this proof using the touching preservation detailed in Proposition \ref{prop:rot-touch}.

\begin{prop}[Supersolution preservation, \cite{CSY}]
    Let $u$ be a viscosity supersolution of $F(D^2u)=\Theta$ on domain $\Omega\subset\re^n$ with $D^2u\ge-\tan(\theta)I$, $-(n-2)\pi/2<\Theta<\pi/2$, and $\theta=(\pi/2-\Theta)/(n-1)\in(0,\pi/2)$.  For any $0<\phi<\pi/2-\theta$, the $\phi$-rotation \eqref{rot} is a viscosity supersolution of $F(D^2\bar u)=\Theta-n\phi$ on the domain $\pd (su+c|x|^2/2)(\Omega)$.
\end{prop}

\begin{proof}
    Let $q(\bar x)$ be a quadratic which touches $\bar u$ at $\bar x_0\in \pd(su+c|x|^2/2)(x_0)$ from below nearby.  We suppose that $x_0=\bar x_0=u(0)=0$.  In this case, $q_\ep=q-\ep|\bar x|^2/2$ also touches $\bar u$.  Since $D^2q_\ep\le (\tan(\pi/2-\phi)-\ep)I$, rotating upwards, or using angle $-\phi$ in \eqref{rot}, is possible and gives a quadratic polynomial $Q_\ep$ with $D^2Q_\ep<\infty$ and $\bar{Q_\ep}=q_\ep$.  By Part 3 in Proposition \ref{prop:rot-touch}, it follows that $Q_\ep$ touches $u$ from below at $x=0$, so $u$ being a supersolution implies $F(D^2Q_\ep)\le \Theta$.  Applying the Hessian transformation rule \eqref{angles}, we find that
    $$
    F(D^2q-\ep I)=F(D^2\bar{Q_\ep})=F(D^2Q_\ep)-n\phi\le \Theta-n\phi.
    $$
    Sending $\ep\to 0$ gives the conclusion.
\end{proof}

What makes the above argument work is that quadratics touching $\bar u$ from below are still quadratics after undoing the rotation, so the original equation informs on the Hessian of the quadratic.

\medskip
The subsolution preservation under rotation can be false.  The Lipschitz example by \cite{MS} is $C^{2,1}$ in rotated space, but non-minimal inside the gradient image of the singularity.  In this situation, the quadratics touching from above rotate back to Lipschitz cones.  Therefore, the original equation is not able to inform on the Hessian of the quadratic.  In \cite{CSY}, the viscosity solutions were assumed convex to effect this rotation.

\medskip
The main result of this section is to show the preservation of subsolutions under the conditions in Theorem \ref{thm:reg}: semi-convex with almost negative phase.  The idea is that if one angle $\theta_1$ of the touching quadratic is maximal in rotated space, then the inequality $\sum_i\theta_i\ge \Theta-n\phi$ can be satisfied without precisely understanding the other angles $\theta_2,\dots,\theta_n$.  What is needed is for the other angles to not be very negative, and $\Theta$ to not be too positive.  This argument turns out to be sharp, surprisingly.

\begin{prop}[Subsolution preservation]
\label{prop:rot}
    Let $u$ be a viscosity subsolution of $F(D^2u)=\Theta$ on open subset $\Omega\subset\re^n$ with $D^2u\ge -(\tan\theta)I$, $\Theta\in(-(n-2)\pi/2,\pi/2)$, and $\theta=(\pi/2-\Theta)/(n-1)\in(0,\pi/2)$.  Then for any $0<\phi<\pi/2-\theta$, the $\phi$-rotation $\bar u$, in \eqref{rot}, is a viscosity subsolution of $F(D^2\bar u)=\Theta-n\phi$, on the open set $\pd(su+c|x|^2/2)(\Omega)$.
\end{prop}

\begin{proof}
    \textbf{Step 1 (singular points):} Let $\bar P$ be a quadratic polynomial touching $\bar u$ at $\bar x_0\in \pd (su+c|x|^2/2)(x_0)$ from above nearby.  Suppose that the largest angle $\tan^{-1}\lambda_1(D^2\bar P)$ is maximal, or $\tan^{-1}\lambda_1(D^2\bar P)\ge \pi/2-\phi$.  Then
    \es{
    F(D^2\bar P)\ge \pi/2-\phi+\sum_{i=2}^n \tan^{-1}\lambda_i(D^2\bar P).
    }
    A problem is whether the other eigenvalues are $-\pi/2$.  Fortunately, after rotation, $\bar u$ is still semi-convex, $D^2\bar u\ge -\tan(\theta+\phi)I$.  Since $\bar P$ touches $\bar u$ from above, it also satisfies $D^2 \bar P\ge -\tan(\theta+\phi)I$.  We obtain
    \es{
    F(D^2\bar P)&\ge \pi/2-\phi+\sum_{i=2}^n (-\theta-\phi)\\
    &=\pi/2-(n-1)\theta-n\phi\\
    &=\Theta-n\phi,
    }
    provided that $\theta=(\pi/2-\Theta)/(n-1)$.
    
\smallskip
    \textbf{Step 2 ($C^{1,1}$ points):} Let $p$ be a quadratic polynomial touching $\bar u$ at $\bar x_0\in\pd(su+c|x|^2/2)(x_0)$ from above nearby.  If $D^2p<\tan(\pi/2-\phi)I$, then rotating it by angle $-\phi$ in \eqref{rot} is possible, and gives a quadratic polynomial $P$ satisfying $\bar P=p$ and $D^2P<\infty$.  By Part 2 in Proposition \ref{prop:rot-touch}, it follows that $P$ touches $u$ at $x_0$ from above nearby.  Since $u$ is a subsolution, we get $F(D^2P)\ge \Theta$.  The angle transformation rule \eqref{angles} gives $F(D^2p)\ge \Theta-n\phi$.
\end{proof}

\begin{rem}
\label{rem:codim}
    If the singular set is assumed to have higher codimension $k$, then the allowed phase and semi-convexity for subsolution preservation changes.  Namely, suppose that either a point $\bar x_0$ is singular with codimension $k$, with $\tan^{-1}\lambda_i(D^2\bar P)\ge \pi/2-\phi$ for $1\le i\le k$ and any touching polynomial $\bar P$ from above, or it is $C^{1,1}$ there, with $\tan^{-1}\lambda_i(D^2p)<\pi/2-\phi$ for all $i$ and some touching $p$ from above.  Then $F(D^2\bar P)\ge \Theta-n\phi$ provided that $\theta= (k\pi/2-\Theta)/(n-k)$.  For $0<\theta<\pi/2$, we are instead allowed $-(n-2k)\pi/2<\Theta<k\pi/2$.  We note that by a similar argument to \eqref{altEq}, smaller semi-convexity $\theta<(k\pi/2-\Theta)/(n-k)$ implies that such $u$ is already $C^{1,1}$.  
    
    The examples from Section \ref{Ex1Proof} have codimension one singularities; it is possible that they can be generalized to codimension $k$, to show that these conditions for subsolution preservation are optimal.  
    
    The regularity question for these solutions appears more subtle.  For $\theta=(k\pi/2-\Theta)/(n-k)$ and $k\ge 2$, one can show that the rotated phase $\bar\Theta>-(n-k)\pi/2$ is subcritical.  In low dimension $n=3,4$, the flatness of graphical special Lagrangian cones \cite{HNY,NV2,Y1} would give regularity of the gradient graph.  This appears unclear in higher dimension $n\ge 5$.
\end{rem}

\begin{rem}
    The phase and semi-convexity conditions in Proposition \ref{prop:rot} are sharp for subsolution preservation, by the examples of Section \ref{Ex1Proof}.  However, more conditions rule out those examples, such as $u$ being a $C^1$ viscosity solution.  It is an interesting question about whether such subsolutions are preserved, hence if they are minimal.  The examples from \cite{NV1,WdY1} are $C^{1,\alpha}$ and minimal.
    
    Yet another condition arises from expanding the viscosity testing functions from $C^2$ to $C^1$.  We say that $u\in C^0$ is a $C^1$-viscosity supersolution of $F(D^2u)=f(x)$ at $x_0$ if $F(D^2v)\le f(x)$ for any polyhomogeneous-at-$x_0$ function $v\in C^1$ which touches $u$ from below at $x_0$, where $F(D^2v(x))$ is interpreted only at the points $x$ where $v$ is twice differentiable.  One can define $C^{1,\alpha}$-viscosity solutions similarly.  

    In fact, the examples of Section \ref{Ex1Proof} are not $C^1$-viscosity supersolutions, which shows that the Dirichlet problem is not solvable in the class of $C^1$-viscosity solutions.  To briefly illustrate the idea, we note that in dimension one, the function $u(x_1)=|x_1|+|x_1|^{3/2}$ is a viscosity solution of $F(D^2u)=\pi/2-\arctan(C|x|^{1/2})=:f(x_1)$.  However, $v(x_1)=|x_1|^{3/2}+x_1^2/(2\ep)$ touches $u$ from below at $x_1=0$ for any $\ep>0$, and we have $F(D^2v)>f(x_1)$ away from the origin.  The examples in Section \ref{Ex1Proof} have a related issue.
\end{rem}

\subsection{Higher regularity by convexity of rotated equation}
\label{sec:regrot}

The rotation $\bar u$ solves $F(D^2\bar u)=\Theta-n\phi$ on the gradient image $\bar\Omega=\pd(su+c|x|^2/2)(\Omega)$.  The gradient graph $(\bar x,D\bar u(\bar x))$ is a Lipschitz submanifold of $\re^n\times\re^n$, and by the calibration argument, it is volume minimizing \cite{HL}.  We will show it is analytic.

\medskip
We can choose $\phi$ large enough for this phase to be negative supercritical.  Let
$$
\phi=\pi/2-\theta-\delta=\frac{(n-2)\pi/2+\Theta}{n-1}-\delta,
$$
for some small $0<\delta<\pi/2-\theta$ to be chosen.  The rotated phase satisfies
\es{
\Theta-n\phi&=\Theta-(n-1)\phi-\phi\\
&=-(n-2)\pi/2+(n-1)\delta-\phi\\
&=-(n-2)\pi/2+n\delta-\frac{(n-2)\pi/2+\Theta}{n-1}.
}
The last term is strictly negative.  Choosing
\es{
\delta=\frac{(n-2)\pi/2+\Theta}{2n(n-1)}=(\pi/2-\theta)/(2n)\in(0,\pi/(4n)),
}
it follows that $0<\delta<\pi/2-\theta$, so this is well defined, and 
\eqal{
\label{newTheta}
\Theta-n\phi=-(n-2)\pi/2-n\delta,
}
so the phase is negative supercritical.  We note that $-\bar u$ viscosity solves $F(D^2(-\bar u))=(n-2)\pi/2+n\delta$, which has convex superlevel set by Yuan \cite{Y3}.  Since $-\bar u$ is $C^{1,1}$, it follows from Evans-Krylov that $-\bar u$ is $C^{2,\alpha}$, with higher $C^k$ regularity by Schauder estimates, and analyticity by Morrey \cite{Mo}.  It follows that the gradient graph $(\bar x,D\bar u(\bar x))$ over $\bar \Omega$ is an analytic, volume minimizing submanifold.

\subsection{Regularity of viscosity solution by constant rank theorem}
\label{sec:reg}

We find a new maximum principle for the maximum eigenvalue of the rotated equation.  At a blow-up point, the other $n-1$ eigenvalues must equal the semi-convex lower bound.  By the convexity of the equation, these other eigenvalues must be constant, so we conclude that the maximum eigenvalue is constant.  This maximum principle is thus indirect.  It does not follow from the constant rank theorem, since the original PDE is saddle, so the rotated equation is not inverse convex.

\smallskip
To facilitate Step 2 below, we take $\Omega=B_1(0)\subset\re^n$, with $\bar\Omega=\pd(su+c|x|^2/2)(\Omega)$.  By \cite[Lemma 3.1]{CSY}, $\bar\Omega$ is open and connected.

\medskip
\textbf{Step 1 (saturation of lower bound at blow-up point).} To the contrary, suppose there is a point $\bar x_0\in\bar\Omega$ for which the maximum eigenvalue $\bar\lambda_1$ of $D^2\bar u(\bar x_0)$ is maximal: $\tan^{-1}\bar\lambda_1=\pi/2-\phi$.  Then the equation $F(D^2\bar u)=\Theta-n\phi$ becomes
\es{
\pi/2-\phi+\sum_{i=2}^n\tan^{-1}\bar\lambda_i=\Theta-n\phi.
}
By semi-convexity $D^2\bar u\ge -\tan(\theta+\phi)I$, and $\theta=(\pi/2-\Theta)/(n-1)$, we obtain
\es{
\Theta-n\phi&\ge \pi/2-\phi -(n-1)(\theta+\phi)\\
&=\Theta-n\phi.
}
This means $\tan^{-1}\bar \lambda_i=-\theta-\phi$ for $2\le i\le n$, such that the semi-convex lower bound is saturated at $\bar x_0$.

\medskip
\textbf{Step 2 (constancy of smaller eigenvalues).} We reflect and define $v=-u$.  Then by oddness and \eqref{newTheta}, $v$ solves
$$
F(D^2v)=-(\Theta-n\phi)=(n-2)\pi/2+n\delta.
$$
This is a positive supercritical phase.  By \cite{Y3}, the superlevel set is convex, so the convex combinations $\lambda_1+\lambda_2+\cdots+\lambda_k$ of the largest eigenvalues of $D^2v$ are viscosity subsolutions of the linearized operator.  For $k\le n-1$, they achieve their maximum of $k\tan(\theta+\phi)$ at $\bar x_0$, so they must be constant, by the strong maximum principle.  We conclude that for $\bar u$, the eigenvalues $\bar\lambda_i\equiv \tan(-\theta-\phi)$ are constant, for $2\le i\le n$.

\medskip
\textbf{Step 3 (maximum principle).} By $F(D^2\bar u)=\Theta-n\phi$, we find that $\tan^{-1}\bar\lambda_1\equiv \pi/2-\phi$ on $\bar\Omega$.  This contradicts Alexandrov's theorem for semi-convex viscosity solution $u(x)$, since for almost every point $x_1\in\Omega$, we have $u$ twice differentiable at $x_1$, with $D^2u(x_1)<\infty$, so $\bar\lambda_1(\bar x_1)<\pi/2-\phi$.  

\medskip
\textbf{Step 4 (regularity of $u$).} Since $D^2\bar u(\bar x)< \tan(\pi/2-\phi)$ at each point $\bar x\in \pd (su+c|x|^2/2)(\Omega)$, Part 4 in Proposition \ref{prop:rot-touch} shows that $D^2u(x)<\infty$ near each point in $\Omega$.  By the continuity of $D^2\bar u$ and the strong convexity of $su+c|x|^2/2$, this is uniform on compact subsets away from $\pd \Omega$.  Moreover, Part 5 shows that $u$ is twice differentiable everywhere in $\Omega$.


\smallskip
To confirm that $u$ is actually smooth, we recall the inverse transformation rule
\eqal{
\label{hess}
D^2u(x)=(sI+cD^2\bar u(\bar x))(cI-sD^2\bar u(\bar x))^{-1},
}
as well as the set-valued rotation map, which, since $u\in C^{1,1}_{loc}$, is now a locally Lipschitz single-valued function:
$$
\bar x(x)=cx+sy=cx+sDu(x).
$$
Because $D^2\bar u(\bar x)<\cot(\phi)I$ uniformly near any point, the rational matrix map $M\mapsto (sI+cM)(cI-sM)^{-1}$ in \eqref{hess} is well defined and analytic nearby $D^2\bar u(\bar x)$.  Its composition with $D^2\bar u(\bar x(x))$ is also a Lipschitz function of $x$, since $\bar x(x)$ is Lipschitz and $\bar u(\bar x)$ is analytic.  We conclude that $D^2u(x)$ is locally Lipschitz, hence $u\in C^{2,1}$.  Iterating this argument, we conclude that $u\in C^\infty$, then analytic by \cite{Mo}.

\subsection{Liouville theorem}

We prove Theorem \ref{thm:liou}.  Let $\Omega=\re^n$.  By strong convexity, $\pd (su+c|x|^2/2)$ is distance increasing, so $\pd(su+c|x|^2/2)(\re^n)=\re^n$.  This means $\bar u$ is an entire solution of $F(D^2\bar u)=\Theta-n\phi$ with Hessian bounds
$$
\tan(-\theta-\phi)I\le D^2\bar u\le \tan(\pi/2-\phi) I.
$$  
As in section \ref{sec:regrot}, this equation is the negative supercritical special Lagrangian equation, and the Evans-Krylov theorem is valid, for this, now, uniformly elliptic PDE with convex sublevel set \cite{Y3}.  We obtain
$$
[D^2\bar u]_{C^{\alpha}(B_{R/2}(0))}\le \frac{C(n,\Theta,\phi,\|D^2\bar u\|_{L^\infty(B_R(0))})}{R^\alpha}\to 0
$$
as $R\to \infty$.  We conclude that $\bar u$ is a quadratic polynomial.  Its gradient graph $(\bar x,D\bar u(\bar x))$ is a plane, so its counterclockwise rotation $(x,Du(x))$ by $\phi$ is as well.  This completes the proof.

\subsection{Effective estimate}
We begin with a few preliminary results. The first says that given a compact continuous curve, one can find a chain of pairwise disjoint and sequentially tangent balls of the same size centered on the curve that connect the endpoints.

\begin{prop}\label{CurveCovering}
Let $\gamma : [0,\,1] \rightarrow \mathbb{R}^n$ be continuous, and let $r > 0$. There exist $k \in \mathbb{N} \cup \{0\}$ and numbers $0 = t_0 < ... < t_k < 1$ such that
\begin{enumerate}
\item $\{B_r(\gamma(t_i))\}_{i = 0}^k$ are pairwise disjoint,
\item $\partial B_r(\gamma(t_i)) \cap \partial B_r(\gamma(t_{i+1})) \neq \emptyset$ for all $0 \leq i < k$, and
\item $\gamma(1) \subset \overline{B_{2r}(\gamma(t_k))}$.
\end{enumerate}
\end{prop}
\begin{proof}
Let $t_0 := 0$, and for $i \geq 1$ define $t_i$ inductively by
$$t_{i} := \sup\{t \in [0,\,1] : |\gamma(t) - \gamma(t_{i-1})| \leq 2r\}.$$
It is obvious that $t_i$ are non-decreasing, that $t_i < t_{i+1}$ whenever $t_i < 1$, and that
$$|\gamma(t_{i+1}) - \gamma(t_i)| = 2r \text{ whenever } t_{i+1} < 1.$$
Moreover, if $t_{i + 1} < 1$, then $|\gamma(t) - \gamma(t_i)| > 2r$ for all $t > t_{i+1}$. We conclude that if $t_{i+1} < 1$, then
$$B_r(\gamma(t_j)) \cap B_r(\gamma(t_i)) = \emptyset$$
for all $j > i$. This establishes properties (1) and (2) for any $k$ such that $t_k < 1$.

Suppose that $t_K < 1$. By the disjointness of the collection $\{B_r(\gamma(t_i))\}_{i = 0}^K$ we have that
$$(K+1)|B_r| \leq |B_{\sup_{t \in [0,\,1]}|\gamma(t) - \gamma(0)| + r}|.$$
In particular, there exists a first number $k$,
$$0 \leq k \leq \left(1 + \frac{\sup_{t \in [0,\,1]}|\gamma(t) - \gamma(0)|}{r}\right)^n - 1,$$
such that $t_k < t_{k+1} = 1$.
\end{proof}

The second is a version of the weak Harnack inequality.

\begin{lem}\label{HarnackChain}
Assume that $a_{ij}(x)w_{ij} \leq 0$ in $B_{4r}(x_1) \subset \mathbb{R}^n$, where $a_{ij}$ has ellipticity constants $0 < \lambda \leq \Lambda < \infty$ and $w \geq 0$. Assume further that $|x_2 - x_1| \leq 2r$. Then
$$\int_{B_r(x_1)} w^p \leq C\int_{B_r(x_2)} w^p$$
for some $C(n,\,\lambda,\,\Lambda) > 0$ and $p(n,\,\lambda,\,\Lambda) > 0$.
\end{lem}
\begin{proof}
By the weak Harnack inequality (see e.g. Theorem 9.22 in \cite{GT} or Theorem 4.8 in \cite{CC}) we have for universal $C,\,p > 0$ that
$$\int_{B_{3r}(x_1)}w^p \leq C|B_r|\inf_{B_{3r}(x_1)} w^p.$$
The right hand side is bounded above by $C|B_r|\inf_{B_r(x_2)} w^p \leq C\int_{B_r(x_2)}w^p$.
\end{proof}

\begin{proof}[{\bf Proof of Effective Bound}:]
We fix $\phi$ as above so that the phase of the equation for $\bar{u}$ is negative supercritical.
We call a constant universal if it depends only on $n$ and $\Theta \in (-(n-2)\pi/2,\, \pi/2)$, and we let $c,\,C$ denote small and large positive universal constants that may change from line to line. We also let
$$L := \|Du\|_{L^{\infty}(B_1)}.$$

The map $\bar{x}$ is distance-increasing with universal lower bound (recall that $\bar{x}$ is the gradient of $\cos(\phi)|x|^2/2 + \sin(\phi)u$, and the Hessian of this function has a positive universal lower bound), so
$$B_c(y) \subset \bar{x}(B_1)$$
for all $y \in \bar{x}(B_{1/2})$. Let $x \in B_{1/2}$, and apply Proposition \ref{CurveCovering} to the curve 
$$\gamma(t) := \bar{x}(tx)$$ 
with $r := c/4$, getting a chain of pairwise disjoint and sequentially tangent balls $\{B_{c/4}(y_i)\}_{i = 0}^k,\, y_i \in \bar{x}(B_{1/2}),\, y_0 = \bar{x}(0),\, |\bar{x}(x) - y_k| \leq c/2$. If $w$ is a nonnegative supersolution to the linearized special Lagrangian equation at $\bar{u}$ (which has universal ellipticity constants) in $\bar{x}(B_1)$ then we can apply Lemma \ref{HarnackChain} and the weak Harnack inequality (with $r = c/4$) to get
$$\int_{B_{c/4}(\bar{x}(x))} w^p \leq C\int_{B_{c/4}(y_k)} w^p \leq ... \leq C^{k+1}\int_{B_{c/4}(\bar{x}(0))}w^p \leq C^{k+2}w^p(\bar{x}(0))$$
with $p > 0$ universal. Thus,
$$\int_{B_{c/4}(y)} w^p \leq C^kw^p(\bar{x}(0))$$
for all $y \in \bar{x}(B_{1/2})$. By Vitali's covering lemma, there exist $N$ disjoint balls of radius $c/12$ centered in $\bar{x}(B_{1/2})$ whose three-times dilations cover $\bar{x}(B_{1/2})$. Thus,
$$\int_{\bar{x}(B_{1/2})} w^p \leq NC^kw^p(\bar{x}(0)).$$

We now claim that
\begin{equation}\label{VolBound}
|\bar{x}(B_1)| \leq C(1+L).
\end{equation}
Indeed, by Lemma \ref{Lambda2}, all but one eigenvalue of $D^2u$ is universally bounded. In particular,
$$\det D\bar{x} \leq C(C + \Delta u).$$
The estimate (\ref{VolBound}) follows from this and the area formula.

It follows from (\ref{VolBound}) that $N \leq C(1+L)$ and that $k \leq C(1+L)$, thus
\begin{equation}\label{GeneralWeakHarnack}
\int_{\bar{x}(B_{1/2})} w^p \leq e^{C(1 + L)}w^p(\bar{x}(0)).
\end{equation}

We now conclude. Integrating $\Delta u$ by parts in $B_{1/2}$ and using the assumed lower bound $D^2u \geq -\tan(\theta)I$ gives
$$\int_{B_{1/2}} (\lambda_1)_+ \leq C(1+L).$$ 
We thus have by Chebyshev's inequality that
$$|\{\lambda_1 < C(1+L)\} \cap B_{1/2}| \geq |B_{1/2}|/2$$
for $C$ large universal. Since $\bar{x}$ is distance-increasing with universal lower bound, we conclude that 
$$\cot^{-1}(\bar{\lambda}_1)-\phi = \cot^{-1}(\lambda_1) \geq c/(1+L)$$ 
on a set of positive universal measure in $\bar{x}(B_{1/2})$. On this same set we must have
$$w := \sum_{i > 1} [\bar{\lambda}_i + \tan(\theta + \phi)] \geq c/(1+L).$$
Here and below we use that
$$\sum_{i > 1} [\tan^{-1}(\bar{\lambda}_i) + (\theta + \phi)] = \cot^{-1}(\bar{\lambda}_1) - \phi,$$
and that the terms in the sum on the left are nonnegative. We conclude using that $w$ is a nonnegative supersolution to the linearized equation (recall that the phase is negative supercritical) and (\ref{GeneralWeakHarnack}) that
$$[c/(1+L)]^p \leq e^{C(1 + L)}w^p(\bar{x}(0)) \quad \Longrightarrow \quad e^{-C(1 + L)} \leq w(\bar{x}(0)).$$
Thus
$$\cot^{-1}(\lambda_1(0)) = \cot^{-1}(\bar{\lambda}_1(\bar{x}(0))) - \phi \geq e^{-C(1+L)},$$
hence
$$\lambda_1(0) \leq e^{C(1 + L)}.$$
Estimates for the higher-order derivatives of $u$ can be obtained by differentiating the relation (\ref{hess}), using the bound for $D^2u(0)$, and using the fact that 
$$|D^k\bar{u}(\bar{x}(0))| \leq C(n,\,k,\,\Theta).$$
This last inequality comes from the universal $C^{1,\,1}$ bound for $\bar{u}$, Evans-Krylov, and Schauder estimates.
\end{proof}

\begin{rem}
The exponential dependence of our estimate on $\|Du\|_{L^{\infty}(B_1)}$ cannot be improved, see Section \ref{Ex2Proof}.
\end{rem}

\begin{rem}
Non-quantitative a priori interior estimates follow quickly from a compactness argument exploiting the regularity of viscosity solutions (above) and Savin's small perturbations theorem (\cite{S}).
\end{rem}

\begin{rem}
\label{rem:sharp}
Our method gives improved effective Hessian estimates for solutions to (\ref{eq}) if $u$ is convex, or if $\Theta$ is supercritical. More precisely, one can bound the volume of the image of the relevant map $\bar{x}$ in those cases by $\|Du\|_{L^{\infty}}^{n-1}$, which by our method gives Hessian bounds that depend exponentially on $\|Du\|_{L^{\infty}}^{n-1}$. Previous estimates depend exponentially on $\|Du\|_{L^{\infty}}^{2(n-1)}$ (\cite{WdY2}). 

In our method, the estimate is in fact dictated by the number of balls $k$ in the Harnack chain. One could thus improve further to exponential dependence on $\|Du\|_{L^{\infty}}$ in those cases if every pair of points in the image of $\bar{x}$ could be connected by a curve in the image of length $\sim \|Du\|_{L^{\infty}}$ (e.g. if the image were star-shaped). The examples in Section \ref{Ex2Proof} show that such an estimate (exponential in $\|Du\|_{L^{\infty}}$) would be sharp.
\end{rem}

\section{Proof of Theorem \ref{thm:Ex} (Sharpness of assumptions)}\label{Ex1Proof}

{\bf Step 1.} Fix arbitrary parameters $\lambda > 0$ and $a_i \neq 0,\, i \geq 4$. Let
$$\Phi(x) := \frac{\lambda x_1^2}{2(1+x_3)} + \frac{\lambda x_2^2}{2(1-x_3)} + \sum_{i \geq 4} [a_ix_i^2/2 + x_i^4/12].$$
Then $D^2\Phi$ has rank $n-1$, since the first two terms are translations of one-homogeneous functions of two variables (thus have Hessians of rank $1$) and the remaining terms split off. We also have the following key property:

\begin{lem}\label{NonDegen}
$F(D^2\Phi)$ has a non-degenerate local minimum at $0$.
\end{lem}

This was shown in the case $n = 3,\, \lambda > 0$ small in Lemma 2.1 of \cite{MS}. However, we will need this result for all $\lambda > 0$ to demonstrate the optimality of Theorem \ref{thm:reg}.

\begin{proof}
The eigenvalues of $D^2\Phi$ are
$$(h(x) + g(x),\, h(x) - g(x),\, 0,\, a_4 + x_4^2,\, ...,\, a_n + x_n^2),$$
where for $a := (1+x_3)^{-1}$ and $b := (1-x_3)^{-1}$ we have
$$h(x) = \lambda\left(\frac{1}{1-x_3^2} + \frac{a^3}{2}x_1^2 + \frac{b^3}{2} x_2^2\right) \text{ and }$$
$$g(x) = \lambda\left(\frac{1}{4} (2abx_3 + (b^3x_2^2 - a^3x_1^2))^2 + a^3b^3x_1^2x_2^2 \right)^{1/2}.$$
Let $f(s) := \tan^{-1}(s)$. Using Taylor expansion we have
$$F(D^2\Phi) = 2f(h) + f''(h)g^2 + \sum_{i \geq 4} [f(a_i) + f'(a_i)x_i^2] + O(|x|^4),$$
whence
$$D^2(F(D^2\Phi))(0) = 2f'(h(0))D^2h(0) + f''(h(0))D^2(g^2)(0) + 2 \sum_{i \geq 4} f'(a_i)e_i \otimes e_i.$$
Here we have used that $\nabla h(0) = 0$ and that $g^2 = O(|x|^2)$. 
The first term is diagonal with entries $\frac{2\lambda}{1+\lambda^2}(1,\,1,\,2,\,0,\,...,\,0)$. The second term can be written 
$$-\frac{4\lambda^3}{(1+\lambda^2)^2} e_3 \otimes e_3.$$
We thus just need
$$\frac{4\lambda}{1+\lambda^2} > \frac{4\lambda^3}{(1+\lambda^2)^2} \Leftrightarrow 1 > \frac{\lambda^2}{1+\lambda^2},$$
which is true for all $\lambda > 0$.
\end{proof}

The rest of the construction is essentially the same as in \cite{MS}. We recall the steps for the reader's convenience. 

For the remainder of the section, positive constants depending only on $\lambda,\,\{a_i\}_{i = 4}^n,$ and  $n$ are called universal.

\vspace{2mm}

{\bf Step 2.} For any positive $\epsilon < \epsilon_0$ sufficiently small universal, the connected component $K_{\epsilon}$ of the set
$\{F(D^2\Phi) < F(D^2\Phi(0)) + \epsilon^2 := c_{\epsilon}^*\}$ containing $0$ is analytic and uniformly convex, with diameter bounded by $C\epsilon$, $C$ universal. Let $\nu$ denote the outer unit normal to $\partial K_{\epsilon}$. In a small neighborhood of $\partial K_{\epsilon}$ we solve via Cauchy-Kovalevskaya the equation
$$F(D^2v) = c_{\epsilon}^*, \quad (v,\,v_{\nu})|_{\partial K_{\epsilon}} = (\Phi,\,\Phi_{\nu})|_{\partial K_{\epsilon}}.$$
Finally, we let
$$w = \begin{cases}
\Phi \text{ in } K_{\epsilon}, \\
v \text{ outside } K_{\epsilon}.
\end{cases}$$
Since $F(D^2v) = F(D^2\Phi)$ on $\partial K_{\epsilon}$, ellipticity forces $D^2v = D^2\Phi$ on $\partial K_{\epsilon}$, so $w \in C^{2,\,1}$. Thus
\begin{equation}\label{Closeness}
|D^2w - D^2\Phi(0)| \leq C\epsilon, \quad C \text{ universal }
\end{equation}
in a small enough neighborhood of $K_{\epsilon}$. 

We let
$$S := \text{sign}\left(\Pi_{i \geq 4}a_i\right).$$
\begin{lem}\label{DetDrop}
We have $S\det D^2w < 0$ outside $K_{\epsilon}$.
\end{lem}

In particular, $w$ solves the degenerate Bellman equation
$$\max\{F(D^2w) - c_{\epsilon}^*,\, S\det D^2w\} = 0,$$
and the zero eigenvalue of $D^2w$ in $\overline{K_{\epsilon}}$ becomes strictly negative when we step outside $\overline{K_{\epsilon}}$.

\begin{proof}
The proof follows nearly verbatim that of Lemma 2.3 in \cite{MS}. The only modifications are that the function $G := \det D^2w$ satisfies 
$$SG_{\nu\nu}(D^2w(x)) > 0$$
at points on $\partial K_{\epsilon}$ where the zero eigendirection $\xi$ of $D^2w$ is not tangent to $\partial K_{\epsilon}$, and that
$$SG_{\xi\xi} > 0$$
at points on $\partial K_{\epsilon}$ where $\xi \perp \nu$.
\end{proof}

{\bf Step 3.} We have that $\nabla w$ maps a small exterior neighborhood of $\overline{K_{\epsilon}}$ diffeomorphically to a small exterior neighborhood of the smooth compact hypersurface with boundary $\Gamma := \nabla w\left(\overline{K_{\epsilon}}\right)$, which is contained in the paraboloid graph
$$\{(z_1,\,z_2,\,(z_2^2-z_1^2)/(2\lambda),\, z_4,\,...,\,z_n)\}.$$
The proof is identical to Lemma 2.4 in \cite{MS}, up to replacing the map $H$ there by
$$H(x) := (w_1,\,w_2,\,x_3,\,w_4,\,...,\,w_n)$$
and noting that the sign of $\det DH$ is $S$.

{\bf Step 4.} Finally, we let $u = -w^*$, i.e.
$$u(\nabla w(x)) = w(x) - x \cdot \nabla w(x),$$
for $x$ in a small neighborhood of $\overline{K_{\epsilon}}$. Away from $\Gamma$, the function $u$ is analytic and satisfies
$$D^2u = -(D^2w)^{-1},$$
hence
\begin{equation}\label{uEqn}
F(D^2u) = c_{\epsilon}^* + \frac{\pi}{2}\left[2-n+2(\# a_i < 0)\right].
\end{equation}
Moreover, $D^2u$ has one eigenvalue that is close to $+ \infty$, while by (\ref{Closeness}) the remaining eigenvalues are within $C\epsilon$ of $(-1/\lambda,\,-1/\lambda,\,-1/a_4,\,...,\,-1/a_n),\, C$ universal. Moreover, we have exactly as in pages 2936-2937 of \cite{MS} that $u_{33} > 0$ tends to $+\infty$ on $\Gamma$, and that $u$ has an upwards Lipschitz singularity on interior points of $\Gamma$. In particular, $u$ is a viscosity subsolution of (\ref{uEqn}) in a neighborhood of $\Gamma$. To see that it is a viscosity supersolution, note as in pages 2936-2937 of \cite{MS} that $u$ is the uniform limit of the functions $-(w - x_3^2/k)^*$, which are analytic supersolutions of (\ref{uEqn}) in a uniform neighborhood of $\Gamma$.

{\bf Step 5.} To prove Theorem \ref{thm:Ex} we play with the parameters $\lambda,\,a_i,\,\epsilon$. 

We first consider the case $\Theta \in [\pi/2,\, (n-2)\pi/2),\, n \geq 4$. Let all $a_i = A < 0$ to be chosen. Then for any choice of $\lambda > 0$ we have 
$$F(D^2u) = (n-4)\frac{\pi}{2} + 2\tan^{-1}(\lambda) + (n-3) \tan^{-1}(A) + \epsilon^2.$$ 
For $0 < \delta < [(n-2)\pi/2 - \Theta]/2$ arbitrarily small, we can take $\lambda$ so large that $2\tan^{-1}\lambda = \pi - \delta$. To ensure that $F(D^2u) = \Theta$ we then need to take $A$ and $\epsilon$ such that
$$(n-3)\tan^{-1}(A) + \epsilon^2 = \Theta -  (n-2)\pi/2 + \delta \in [-(n-3)\pi/2 + \delta,\, -\delta),$$
i.e. for $\tilde{\Theta} := (\Theta -  (n-2)\pi/2 + \delta)/(n-3)$,
$$\tan^{-1}(A) + \epsilon^2/(n-3) = \tilde{\Theta} \in [-\pi/2 + \delta/(n-3),\, -\delta/(n-3)).$$
It is clear that for all $\epsilon > 0$ small we can take $A = \tan(\tilde{\Theta}) + O(\epsilon^2) \in [-c^{-1}(\delta),\, -c(\delta)]$ so that this is satisfied. Since 
$$D^2u \geq -(\lambda^{-1}+C(n,\,\delta)\epsilon)I,$$
we conclude by taking $\delta$, then $\epsilon$ arbitrarily small that when $\Theta \geq \frac{\pi}{2}$ is subcritical, no negative lower bound on the Hessian will lead to a regularity result.

We now consider the case $\Theta \in (-(n-2)\pi/2,\, \pi/2),\, n \geq 3$. We take all $a_i = \lambda > 0$, so that
$$F(D^2u) = (n-1)\tan^{-1}(\lambda) - (n-2)\frac{\pi}{2} + \epsilon^2.$$
To ensure that $F(D^2u) = \Theta$ we need
$$\tan^{-1}(1/\lambda) - \epsilon^2/(n-1) = \theta := (\pi/2-\Theta)/(n-1) \in (0,\,\pi/2).$$
For $\epsilon > 0$ arbitrarily small we can take $-\lambda^{-1} = -\tan(\theta) + O(\epsilon^2)$ such that this holds.
This means that the smallest eigenvalue $\mu = -\lambda^{-1} + O(\epsilon)$ of $D^2u$ satisfies
$$\mu \geq -\tan(\theta) + O(\epsilon).$$
Since $\epsilon > 0$ is arbitrary, this shows that the lower bound required on the Hessian in Theorem \ref{thm:reg} is sharp.


\section{Proof of Theorem \ref{thm:Ex2} (Sharpness of effective bound)}\label{Ex2Proof}

{\bf The case $n = 2,\, \Theta = \pi/2$.} Let $g$ denote the Legendre transform of $\cosh$, i.e. 
$$g(s) = s\sinh^{-1}(s) - \sqrt{1+s^2}.$$
For $M > 1$ let
$$u(x,\,y) := e^{-M}\cos(y)g(e^Mx/\cos(y))$$
in the square $Q := (-1,\,1)^2 \subset \mathbb{R}^2$. Below $C$ will denote a constant independent of $M$. It is straightforward to check by hand that
$$\det D^2u = 1,$$
that
$$|u_x| = |g'(e^Mx/\cos(y))| \leq CM, \quad |u_y| = |\sin(y)|(e^{-2M} + x^2/\cos^2(y))^{1/2} \leq C,$$
and that
$$u_{xx}(0,\,0) = e^M,$$
proving that exponential dependence of Hessian estimates on $\|Du\|_{L^{\infty}}$ is optimal in the case $n = 2,\, \Theta = \pi/2$. For the calculation, the relations
$$sg' - g = (1+s^2)^{1/2}, \quad g'' = (1+s^2)^{-1/2}$$
are useful.

\begin{rem}
The function $u$ is obtained by taking the partial Legendre transform in the $x$ direction of the harmonic function $e^{-M}\cosh(x)\cos(y)$ in a strip of vertical length $\sim 1$ and horizontal length $\sim M$.
\end{rem}

\vspace{3mm}

{\bf The case $n = 2,\, \Theta \neq 0$.} Fix $\theta \in (-\pi/2,\,\pi/2)$, and let
$$s := \sin(\theta), \quad c := \cos(\theta).$$
From the calculation
$$u_{yy} = (e^{-2M}\cos^2(y) + x^2)^{1/2} + x^2\tan^2(y)(e^{-2M}\cos^2(y) + x^2)^{-1/2}$$
we see that there exist $r_1,\, r_2 > 0$ depending only on $\theta$ such that 
$$T(x,\,y) := (x,\, cy + su_y)$$
maps $r_1Q$ diffeomorphically onto a region containing $r_2Q$ for all $M > r_1^{-1}$. After a rotation of $\mathbb{R}^4$ defined by
$$(x,\,y,\,z,\,w) \mapsto (x,\, cy + sw,\, z,\, -sy + cw),$$
the graph $\{(x,\,y,\,u_x,\,u_y)\}$ of $\nabla u$ is represented by the graph over $r_2Q$ of a new potential $\bar{u}$ defined by
$$\nabla \bar{u}(x,\, cy + su_y) = (u_x,\, -sy + cu_y).$$
That is, we take a partial Legendre-Lewy-Wang-Yuan transform of $u$ in the $y$ variable. We claim that
\begin{equation}\label{PartialRot}
F(D^2\bar{u}) = \pi/2 - \theta.
\end{equation}
Since it is patently true that
$$|\nabla \bar{u}| \leq CM,$$
and a short calculation shows that
$$\bar{u}_{11}(0,\,0) = u_{xx}(0,\,0) = e^M,$$
this demonstrates the optimality of exponential dependence of interior Hessian estimates on $\|Du\|_{L^{\infty}}$ for all non-zero phases in $\mathbb{R}^2$. Such estimates were proven by Warren and Yuan in \cite{WY2}.

To verify (\ref{PartialRot}), it easier to check that its other form
$$c(1-\det D^2\bar{u}) = s\Delta \bar{u}$$
is satisfied. This can be done by differentiating the defining expression for $\nabla \bar{u}$ and using that $\det D^2u = 1$.

\vspace{3mm}

{\bf Optimality of effective bound in Theorem \ref{thm:reg}.} Take $\theta \in (0,\, \pi/2)$ and $\bar{u}$ as above. From the equation for $\bar{u}$ it is obvious that $D^2 \bar{u} > -\tan(\theta) I$. Now let
$$w := \bar{u}(x_1,\,x_2) - \tan(\theta) \sum_{i > 2} x_i^2/2$$
in $\mathbb{R}^n$. Then
$$F(D^2w) = \Theta := \pi/2 - (n-1)\theta,$$
and $w$ satisfies the semi-convexity condition we assume in Theorem \ref{thm:reg}. By the behavior of $\bar{u}$ as $M$ gets large, the exponential effective bound (\ref{HessianBound}) is optimal. (Note that by varying $\theta$ we can arrange that $F(D^2w)$ is anything in $(-(n-2)\pi/2,\,\pi/2)$.)

\vspace{3mm}

{\bf The case $u$ convex, $\Theta \geq \pi/2$.} For $\theta \in (-\pi/2,\,0],\, A \geq 0,$ let
$$w := \bar{u}(x_1,\,x_2) + A \sum_{i > 2} x_i^2/2.$$
Then $w$ is convex, and $F(D^2w) = \pi/2 - \theta + (n-2)\tan^{-1}(A)$. By varying $\theta$ and $A$ in the ranges above we can arrange that $F(D^2w)$ is anything in $[\pi/2,\, n\pi/2)$. Again, by the behavior of $\bar{u}$ as $M$ gets large, exponential dependence of Hessian estimates on $\|Du\|_{L^{\infty}}$ is the best one can expect for convex solutions to (\ref{eq}) with phase in $[\pi/2,\, n\pi/2)$ in any dimension $n \geq 2$.






\end{document}